\documentclass{amsart}
\usepackage{amsthm,amssymb,amsmath}
\usepackage{epsfig}
\usepackage{color}

\newtheorem{prelem}{{\bf Proposition}}

 \newtheorem{theorem}{Theorem}
\newtheorem{corollary}[theorem]{Corollary}

\newtheorem{observation}[theorem]{Observation}

\theoremstyle{definition}

\theoremstyle{remark}

\title{The Roman $(k,k)$-domatic number of a graph}
\author {
$^1$A.P. Kazemi,  $^{2}$S.M. Sheikholeslami and $^3$L. Volkmann \vspace{4mm}\\
$^1$Department of Mathematics\\
University of Mohaghegh Ardabili\\
P. O. Box 179, Ardabil, Iran\\
{\tt adelpkazemi@yahoo.com}\vspace{3mm} \\
$^{2}$Department of Mathematics \\
Azarbaijan University of Tarbiat Moallem
\\ Tabriz, I.R. Iran\\
{\tt s.m.sheikholeslami@azaruniv.edu}\vspace{3mm} \\
$^{3}$Lehrstuhl II f\"{u}r Mathematik\\
RWTH Aachen University\\
52056 Aachen, Germany\\
{\tt volkm@math2.rwth-aachen.de}\vspace{4mm} \\
}
\date{}
\begin{document}
\maketitle

\begin{abstract}
Let $k$ be a positive integer. A {\em Roman $k$-dominating function}
on a graph $G$ is a labeling $f:V (G)\longrightarrow \{0, 1, 2\}$
such that every vertex with label 0 has at least $k$ neighbors with
label 2. A set $\{f_1,f_2,\ldots,f_d\}$ of distinct Roman
$k$-dominating functions on $G$ with the property that
$\sum_{i=1}^df_i(v)\le 2k$ for each $v\in V(G)$, is called a {\em
Roman $(k,k)$-dominating family} (of functions) on $G$. The maximum
number of functions in a Roman $(k,k)$-dominating family on $G$ is
the {\em Roman $(k,k)$-domatic number} of $G$, denoted by
$d_{R}^k(G)$. Note that the Roman $(1,1)$-domatic number
$d_{R}^1(G)$ is the usual Roman domatic number $d_{R}(G)$. In this
paper we  initiate the study of the Roman $(k,k)$-domatic number in
graphs and we  present sharp bounds for $d_{R}^k(G)$. In addition,
we determine the Roman $(k,k)$-domatic number of some graphs. Some
of our results extend those given by Sheikholeslami and Volkmann in
2010 for the Roman domatic number. \vspace{4mm}

\noindent{\bf Keywords:} Roman domination number, Roman domatic
number, Roman $k$-domination number, Roman $(k,k)$-domatic number.
\\ {\bf MSC 2000}: 05C69
\end{abstract}
\section{Introduction}
In this paper, $G$ is a simple graph with vertex set $V=V(G)$ and edge set $%
E=E(G)$. The order $|V|$ of $G$ is denoted by $n=n(G)$. For every vertex $%
v\in V$, the \emph{open neighborhood} $N(v)$ is the set $\{u\in V(G)\mid
uv\in E(G)\}$ and the \emph{closed neighborhood} of $v$ is the set $N[v] =
N(v) \cup \{v\}$. The \emph{degree} of a vertex $v\in V(G)$ is $%
\deg_G(v)=\deg(v)=|N(v)|$. The \emph{minimum} and \emph{maximum degree} of a
graph $G$ are denoted by $\delta=\delta(G)$ and $\Delta=\Delta(G)$,
respectively. The \emph{open neighborhood} of a set $S\subseteq V$ is the
set $N(S)=\cup_{v\in S}N(v)$, and the \emph{closed neighborhood} of $S$ is
the set $N[S]=N(S)\cup S$. The complement of a graph $G$ is denoted by $%
\overline{G}$. We write $K_n$ for the \emph{complete graph} of order $n$ and
$C_n$ for a \emph{cycle} of length $n$. Consult \cite{hhs, W} for the
notation and terminology which are not defined here.

Let $k$ be a positive integer. A subset $S$ of vertices of $G$ is a \emph{$k$%
-dominating set~} if $|N_{G}(v)\cap S|\geq k$ for every $v\in V(G)-S$. The
\emph{$k$-domination number} $\gamma _{k}(G)$ is the minimum cardinality of
a $k$-dominating set of $G$. A \emph{$k$-domatic partition }is a partition
of $V$ into $k$-dominating sets, and the \emph{$k$-domatic number} $d_{k}(G)$
is the largest number of sets in a $k$-domatic partition. The $k$-domatic
number was introduced by Zelinka \cite{Z}. Further results on the $k$%
-domatic number can be found in the paper \cite{KV2} by K$\ddot{\mathrm{a}}$%
mmerling and Volkmann.

Let $k\geq 1$ be an integer. Following K$\ddot{\mathrm{a}}$mmerling and
Volkmann \cite{KV}, a \emph{Roman $k$-dominating function} (briefly RkDF) on
a graph $G$ is a labeling $f:V(G)\rightarrow \{0,1,2\}$ such that every
vertex with label 0 has at least $k$ neighbors with label 2. The $\emph{weight%
}$ of a Roman $k$-dominating function is the value $f(V(G))=\sum_{v\in
V(G)}f(u)$. The minimum weight of a Roman $k$-dominating function on a graph
$G$ is called the \emph{Roman $k$-domination number}, denoted by $\gamma
_{kR}(G)$. Note that the Roman $1$-domination number $\gamma _{1R}(G)$ is
the usual Roman domination number $\gamma _{R}(G)$. A $\gamma _{kR}(G)$-%
\emph{function} is a Roman $k$-dominating function of $G$ with weight $%
\gamma _{kR}(G)$. A Roman $k$-dominating function $f:V\rightarrow \{0,1,2\}$
can be represented by the ordered partition $(V_{0},V_{1},V_{2})$ (or $%
(V_{0}^{f},V_{1}^{f},V_{2}^{f})$ to refer to $f$) of $V$, where $%
V_{i}=\{v\in V\mid f(v)=i\}$. In this representation, its weight is $\omega
(f)=|V_{1}|+2|V_{2}|$. Since $V_{1}^{f}\cup V_{2}^{f}$ is a $k$-dominating
set when $f$ is an RkDF, and since placing weight 2 at the vertices of a $k$%
-dominating set yields an RkDF, in \cite{KV}, it was observed that
\begin{equation}
\gamma _{k}(G)\leq \gamma _{kR}(G)\leq 2\gamma _{k}(G).  \label{eqq}
\end{equation}

A set $\{f_{1},f_{2},\ldots ,f_{d}\}$ of distinct Roman $k$-dominating
functions on $G$ with the property that $\sum_{i=1}^{d}f_{i}(v)\leq 2k$ for
each $v\in V(G)$ is called a \emph{Roman $(k,k)$-dominating family} (of
functions) on $G$. The maximum number of functions in a Roman $(k,k)$%
-dominating family (briefly R$(k,k)$D family) on $G$ is the \emph{Roman $%
(k,k)$-domatic number} of $G$, denoted by $d_{R}^{k}(G)$. The Roman $(k,k)$%
-domatic number is well-defined and
\begin{equation}
d_{R}^{k}(G)\geq 1  \label{eq0}
\end{equation}%
for all graphs $G$ since the set consisting of any RkDF forms an R$(k,k)$D
family on $G$ and if $k\geq 2$, then
\begin{equation}
d_{R}^{k}(G)\geq 2  \label{eq00}
\end{equation}%
since the functions $f_{i}:V(G)\rightarrow \{0,1,2\}$ defined by $f_{i}(v)=i$
for each $v\in V(G)$ and $i=1,2$ forms an R$(k,k)$D family on $G$ of order 2.

The definition of the Roman dominating function was given implicitly by
Stewart \cite{s} and ReVelle and Rosing \cite{rr}. Cockayne, Dreyer Jr.,
Hedetniemi and Hedetniemi \cite{cdhh} as well as Chambers, Kinnersley,
Prince and West \cite{west} have given a lot of results on Roman domination.

Our purpose in this paper is to initiate the study of Roman $(k,k)$-domatic
number in graphs. We first study basic properties and bounds for the Roman $%
(k,k)$-domatic number of a graph. In addition, we determine the Roman $(k,k)$%
-domatic number of some classes of graphs.\newline

The next known results are useful for our investigations.

\begin{prelem}
\label{V0} \emph{\textbf{(K\"{a}mmerling, Volkmann \cite{KV} 2009)} Let $k\ge 1$ be an
integer, and let $G$ be a graph of order $n$.  If $n\le 2k$, then $\gamma_{kR}(G)=n$.
If $n\ge 2k+1$, then $\gamma_{kR}(G)\ge 2k$.}
\end{prelem}


\begin{prelem}
\label{V1} \emph{\textbf{(K\"{a}mmerling, Volkmann \cite{KV} 2009)} Let $G$
be a graph of order $n$. Then $\gamma_{kR}(G)<n$ if and only if $G$ contains
a bipartite subgraph $H$ with bipartition $X,Y$ such that $|X|>|Y|\ge k$ and
$\deg_H(v)\ge k$ for each $v\in X$.}
\end{prelem}

\begin{prelem}
\label{Delta} \emph{\textbf{(K\"{a}mmerling, Volkmann \cite{KV} 2009)} If $G$
is a graph of order $n$ and maximum degree $\Delta\ge k$, then
\begin{equation*}
\gamma_{kR}(G)\ge\left\lceil\frac{2n}{\frac{\Delta}{k}+1}\right\rceil.
\end{equation*}
}
\end{prelem}

\begin{prelem}
\label{SV} \emph{\textbf{(Sheikholeslami, Volkmann \cite{SV} 2010)}
If $G$ is a graph, then $d_{R}(G)=1$ if and only if $G$ is empty.}
\end{prelem}

\begin{prelem}
\label{1d=n} \emph{\textbf{(Sheikholeslami, Volkmann \cite{SV} 2010)} If $G$
is a graph of order $n\ge 2$, then $d_R(G)=n$ if and only if $G$ is the
complete graph on $n$ vertices.}
\end{prelem}

\begin{prelem}
\label{Knk=1} \emph{\textbf{(Sheikholeslami, Volkmann \cite{SV} 2010)} Let $%
K_n$ be the complete graph of order $n\ge 1$. Then $d_{R}^1(K_n)=1$.}
\end{prelem}

\begin{prelem}
\label{pq} \emph{\textbf{(Sheikholeslami, Volkmann \cite{SV2})} Let $K_{p,q}$ be the complete
bipartite graph of order $p+q$ such that $q\ge p\ge 1$. Then
$\gamma_{kR}(K_{p,q})=p+q$ when $p<k$ or $q=p=k$,
$\gamma_{kR}(K_{p,q})=k+p$ when $p+q\ge 2k+1$ and $k\le p\le 3k$ and
$\gamma_{kR}(K_{p,q})=4k$ when $p\ge 3k$.}
\end{prelem}

We start with the following observations and properties. The first observation
is an immediate consequence of (\ref{eq00}) and Proposition \ref{SV}.

\begin{observation}
\emph{If $G$ is a graph, then $d_{R}^k(G)=1$ if and only if $k=1$ and $G$ is
empty.}
\end{observation}

\begin{observation}
\emph{If $G$ is a graph and $k\ge 2$ is an integer, then $d_{R}^k(G)=2$ if
and only if $G$ is trivial.}
\end{observation}

\begin{proof}
If $G$ is trivial, then obviously $d_{R}^k(G)=2$. Now let $G$ be nontrivial
and let $v\in V(G)$. Define $f,g,h:V(G)\rightarrow \{0,1,2\}$ by
\begin{equation*}
f(v)=1\; \mathrm{and}\; f(x)=2\;\mathrm{if}\; x\in V(G)-\{v\},
\end{equation*}
\begin{equation*}
g(v)=2\; \mathrm{and}\; g(x)=1\;\mathrm{if}\; x\in V(G)-\{v\},
\end{equation*}
and
\begin{equation*}
h(x)=1 \;\mathrm{if}\; x\in V(G).
\end{equation*}
It is clear that $\{f,g,h\}$ is an R$(k,k)$D family of $G$ and hence $d_{R}^k(G)\ge 3$.
This completes the proof.
\end{proof}

\begin{observation}\label{Ob:3}
\label{obs} \emph{If $G$ is a graph and $k\ge \Delta(G)+1$ is an integer,
then $d_{R}^k(G)\le2k-1$.}
\end{observation}

\begin{proof}
If $d_{R}^k(G)=1$, then the statement is trivial. Let $d_{R}^k(G)\ge
2$. Since $k\geq \Delta(G)+1$, we have $\gamma_{kR}(G)=n$. Let
$\{f_1, f_2,\ldots,f_d\}$ be an R$(k,k)$D family on $G$ such that $d= d_{R}^k(G)$.
Since $f_1,f_2,\ldots,f_d$ are distinct, we may assume
$f_i(v)=2$ for some $i$
and some $v\in V(G)$. It follows from $\sum_{j=1}^df_j(v)\le 2k$ that $%
\sum_{j\neq i}f_j(v)\le 2k-2$. Thus $d-1\le 2k-2$ as desired.
\end{proof}

\begin{observation}
\label{obs2}\emph{If $k\ge 2$ is an integer, and $G$ is a graph of order $n\ge 2k-2$, then
$d_R^k(G)\ge 2k-1$.}
\end{observation}
\begin{proof}
If $V(G)=\{v_1,v_2,\ldots,v_n\}$, then define  $f_j:V(G)\rightarrow \{0,1,2\}$ by
$f_j(v_j)=2$ and $f_j(x)=1$ for $x\in V(G)-\{v_j\}$ and $1\le j\le 2k-2$ and
$f_{2k-1}:V(G)\rightarrow \{0,1,2\}$ by $f_{2k-1}(x)=1$ for each $x\in V(G)$. Then
$f_1,f_2,\ldots,f_{2k-1}$ are distinct with $\sum_{i=1}^{2k-1}f_i(x)=2k$ for each $x\in V(G)$.
Therefore $\{f_1,f_2,\ldots,f_{2k-1}\}$ is an R$(k,k)$D family on $G$, and thus $d_R^k(G)\ge 2k-1$.
\end{proof}

The last two observations lead to the next result immediately.
\begin{corollary}
\emph{Let $k\ge 2$ be an integer. If $G$ is a graph of order $n\ge 2k-2$ and $k\ge\Delta(G)+1$, then
$d_R^k(G)=2k-1$.}
\end{corollary}

\begin{observation}
\label{mapping} \emph{Let $k\ge 1$ be an integer, and let $G$ be a graph of order $n$. If
$k\ge 2^n$, then $d_R^k(G)=2^n$.}
\end{observation}
\begin{proof}
Let $\{f_1,f_2,\ldots,f_d\}$ be the set of all pairwise distinct functions from $V(G)$ into the
set $\{1,2\}$. Then $f_i$ is a Roman $k$-dominating function on $G$ for $1\le i\le d$, and it
is well-known that $d=2^n$. The hypothesis $k\ge 2^n$ leads to
$$\sum_{i=1}^df_i(v)\le 2d=2\cdot 2^n\le 2k$$
for each vertex $v\in V(G)$. Therefore  $\{f_1,f_2,\ldots,f_d\}$ is an R$(k,k)$D family on $G$
and thus $d_R^k(G)\ge 2^n$.

Now let $f:V(G)\longrightarrow\{0,1,2\}$ be a Roman $k$-dominating function on $G$. Since
$k\ge 2^n>n>\Delta(G)$, it is impossible that $f(x)=0$ for any vertex $x\in V(G)$. Hence the
number of Roman $k$-dominating functions on $G$ is at least $2^n$ and so $d_R^k(G)\le 2^n$.
This yields the desired identity.
\end{proof}


\begin{observation}
\label{complete} \emph{If $k\ge 1$ is an integer, then $\gamma_{kR}(K_n)=\min\{n,2k\}$.}
\end{observation}
\begin{proof}
If $n\le 2k$, then Proposition \ref{V0} implies that $\gamma_{kR}(K_n)=n$.

Assume now that $n\ge 2k+1$. It follows from Proposition \ref{V0} that $\gamma_{kR}(K_n)\ge 2k$.
Let $V(K_n)=\{v_1,v_2,\ldots,v_n\}$, and define  $f:V(K_n)\rightarrow \{0,1,2\}$ by
$f(v_1)=f(v_2)=\ldots=f(v_k)=2$ and $f(v_j)=0$ for $k+1\le j\le n$.
Then $f$ is an RkDF on $K_n$ of weight $2k$ and thus $\gamma_{kR}(K_n)\le 2k$, and the proof is complete.

\end{proof}


\section{Properties of the Roman ($k,k$)-domatic number}

In this section we present basic properties of $d_{R}^k(G)$ and sharp bounds
on the Roman $(k,k)$-domatic number of a graph.

\begin{theorem}
\label{gammast} \emph{Let $G$ be a graph of order $n$ with Roman $k$%
-domination number $\gamma_{kR}(G)$ and Roman $(k,k)$-domatic number $%
d_{R}^k(G)$. Then
\begin{equation*}
\gamma_{kR}(G)\cdot d_{R}^k(G)\le 2kn.
\end{equation*}
Moreover, if $\gamma_{kR}(G)\cdot d_{R}^k(G)= 2kn,$ then for each R$(k,k)$D
family $\{f_1,f_2,\ldots ,f_d\}$ on $G$ with $d=d_{R}^k(G)$, each function $%
f_i$ is a $\gamma_{kR}(G)$-function and $\sum_{i=1}^df_i(v)=2k$ for all $%
v\in V$. }
\end{theorem}

\begin{proof}
Let $\{f_1, f_2,\ldots,f_d\}$ be an R$(k,k)$D family on $G$ such that $d =
d_{R}^k(G)$ and let $v\in V$. Then
\begin{eqnarray*}
d\cdot \gamma_{kR}(G) & = & \sum_{i=1}^d\gamma_{kR}(G) \\
& \le & \sum_{i=1}^d\sum_{v\in V}f_i(v) \\
& = & \sum_{v\in V}\sum_{i=1}^df_i(v) \\
& \le & \sum_{v\in V} 2k \\
& = & 2kn.
\end{eqnarray*}

If $\gamma_{kR}(G)\cdot d_{R}^k(G)=2kn,$ then the two inequalities occurring
in the proof become equalities. Hence for the R$(k,k)$D family $\{f_1,f_2,
\ldots ,f_d\}$ on $G$ and for each $i$, $\sum_{v\in V}f_i(v)=\gamma_{kR}(G)$%
, thus each function $f_i$ is a $\gamma_{kR}(G)$-function, and $%
\sum_{i=1}^df_i(v)=2k$ for all $v\in V$.
\end{proof}

\begin{theorem}
\label{Th:2} \emph{Let $G$ be a graph of order $n\geq 2$ and $k\geq 1$ be an
integer. Then $\gamma _{kR}(G)=n$ and $d_{R}^{k}(G)=2k$ if and only if $G$
does not contain a bipartite subgraph $H$ with bipartition $X,Y$ such that $%
|X|>|Y|\geq k$ and $\deg _{H}(v)\geq k$ for each $v\in X$ and $G$ has $2k$
or $2k-1$ connected bipartite subgraphs $H_{i}=(X_{i},Y_{i})$ with $%
|X_{i}|=|Y_{i}|$, $\deg _{H_{i}}(v)\geq k$ for each $v\in X_{i}$ and $%
|\{i\mid u\in Y_{i}\}|=|\{i\mid u\in X_{i}\}|=k$ for each $u\in V(G)$.}
\end{theorem}

\begin{proof}
Let $\gamma _{kR}(G)=n$ and $d_{R}^{k}(G)=2k$. It follows from Proposition %
\ref{V1} that $G$ does not contain a bipartite subgraph $H$ with bipartition
$X,Y$ such that $|X|>|Y|\geq k$ and $\deg _{H}(v)\geq k$ for each $v\in X$.
Let $\{f_{1},\ldots ,f_{2k}\}$ be a Roman $(k,k)$-dominating family on $G$.
By Theorem \ref{gammast}, $\gamma _{kR}(G)=\omega (f_{i})=n$ for each $i$.
First let for each $i$, there exists a vertex $x$ such that $f_{i}(x)\neq 1$%
. Let $H_{i}$ be a subgraph of $G$ with vertex set $V_{0}^{f_{i}}\cup
V_{2}^{f_{i}}$ and edge set $E(V_{0}^{f_{i}},V_{2}^{f_{i}})$. Since $\omega
(f_{i})=n$ and $f_{i}$ is a Roman $k$-dominating function, $%
|V_{2}^{f_{i}}|=|V_{0}^{f_{i}}|$ and $\deg _{H_{i}}(v)\geq k$ for each $v\in
V_{0}^{f_{i}}$. By Theorem \ref{gammast}, $\sum_{i=1}^{2k}f_{i}(v)=2k$ for
each $v\in V(G)$ which implies that $|\{i\mid v\in V_{2}^{f_{i}}\}|=|\{i\mid
v\in V_{0}^{f_{i}}\}|=k$ for each $v\in V(G)$. Now let $f_{i}(x)=1$ for each
$x\in V(G)$ and some $i$, say $i=2k$. Define the bipartite subgraphs $H_{i}$
for $1\leq i\leq 2k-1$ as above.

Conversely, assume that $G$ does not contain a bipartite subgraph $H$ with
bipartition $X, Y$ such that $|X|>|Y |\ge k$ and $\deg_H(v)\ge k$ for each $%
v\in X$ and $G$ has $2k$ or $2k-1$ connected bipartite subgraphs $%
H_i=(X_i,Y_i)$ with $|X_i|=|Y_i|$ and $\deg_{H_i}(v)\ge k$ for each $v\in X_i
$. Then by Proposition \ref{V1}, $\gamma_{kR}(G)=n$. If $G$ has $2k$
connected bipartite subgraphs $H_i$, then the mappings $f_i:V(G)\rightarrow
\{0,1,2\}$ defined by
\begin{equation*}
f_i(u)=2\;\mathrm{if}\; u\in Y_i,\;f_i(v)=0\;\mathrm{if}\; v\in X_i,\newline
\mathrm{and}\; f_i(x)=1\;\mathrm{for\; each}\; x\in V-(X_i\cup Y_i)
\end{equation*}
are Roman $k$-dominating functions on $G$ and $\{f_i\mid 1\le i\le 2k\}$ is
a Roman $(k,k)$-dominating family on $G$. If $G$ has $2k-1$ connected
bipartite subgraphs $H_i$, then the mappings $f_i,g:V(G)\rightarrow \{0,1,2\}
$ defined by $g(x)=1 \;\mathrm{for\; each}\; x\in V(G)$ and
\begin{equation*}
f_i(u)=2\;\mathrm{if}\; u\in Y_i,\;f_i(v)=0\;\mathrm{if}\; v\in X_i,\newline
\mathrm{and}\; f_i(x)=1\;\mathrm{for\; each}\; x\in V-(X_i\cup Y_i)
\end{equation*}
are Roman $k$-dominating functions on $G$ and $\{g,f_i\mid 1\le i\le 2k-1\}$
is a Roman $(k,k)$-dominating family on $G$.

Thus $d_{R}^k(G)\ge 2k$. It follows from Theorem \ref{gammast} that $%
d_{R}^k(G)= 2k$, and the proof is complete.
\end{proof}

The next corollary is an immediate consequence of Proposition \ref{Delta}, Observation \ref{Ob:3} and
Theorem \ref{gammast}.

\begin{corollary}
\label{Delta1} \emph{For every graph $G$ of order $n$, $d_{R}^k(G)\leq \max\{\Delta, k-1\} +k$.}

\end{corollary}

Let $A_1\cup A_2\cup\ldots\cup A_d$ be a $k$-domatic partition of $V(G)$
into $k$-dominating sets such that $d=d_k(G)$. Then the set of functions $%
\{f_1,f_2,\ldots,f_d\}$ with $f_i(v)=2$ if $v\in A_i$ and $f_i(v)=0$
otherwise for $1\le i\le d$ is an RkD family on $G$. This shows that $%
d_k(G)\le d_{kR}(G)$ for every graph $G$. Since $\gamma_{kR}(G)\ge
\min\{n,\gamma_k(G)+k\}$ (cf. \cite{KV}), for each graph $G$ of order $n\ge 2
$, Theorem \ref{gammast} implies that $d_{R}^k(G)\le \frac{2kn}{%
\min\{n,\gamma_k(G)+k\}}$. Combining these two observations, we obtain the
following result.

\begin{corollary}
\label{cor1} \emph{For any graph $G$ of order $n$,
\begin{equation*}
d_k(G)\le d_{R}^k(G)\le \frac{2kn}{\min\{n,\gamma_k(G)+k\}}.
\end{equation*}%
}
\end{corollary}

\begin{theorem}
\label{Knk} \emph{Let $K_n$ be the complete graph of order $n$ and $k$ a
positive integer. Then $d_{R}^k(K_n)=n$ if $n\ge 2k$,  $d_{R}^k(K_n)\le 2k-1$ if $n\le 2k-1$
and $d_R^k(K_n)=2k-1$ if $k\ge 2$ and $2k-2\le n\le 2k-1$.}
\end{theorem}

\begin{proof}
By Proposition \ref{Knk=1}, we may assume that $k\geq 2$. Assume that $V(K_{n})=\{x_{1},x_{2},...,x_{n}\}$.
First let $n\geq 2k$. Since Observation \ref{complete} implies that $\gamma _{kR}(K_{n})=2k$, it follows
from Theorem \ref{gammast} that $d_{R}^{k}(K_{n})\leq n$. For $1\leq i\leq n$, define now
$f_{i}:V(K_n)\rightarrow \{0,1,2\}$ by
\begin{equation*}
f_{i}(x_{i})=f_{i}(x_{i+1})=\ldots =f_{i}(x_{i+k-1})=2\;\mathrm{and} \;f_{i}(x)=0\;\mathrm{otherwise},
\end{equation*}
where the indices are taken modulo $n$.
It is easy to see that $\{f_{1},f_2,\ldots ,f_{n}\}$ is an R$(k,k)D$ family on $G$ and hence
$d_{R}^{k}(K_{n})\geq n$. Thus $d_{R}^{k}(K_{n})=n$.

Now let $n\leq 2k-1$. Then Observation \ref{complete} yields $\gamma _{kR}(K_{n})=n$, and it follows from
Theorem \ref{gammast} that $d_{R}^{k}(K_{n})\leq 2k$. Suppose to the
contrary that $d_{R}^{k}(K_{n})=2k$. Then by Theorem \ref{gammast}, each
Roman $k$-dominating function $f_{i}$ in any R$(k,k)$D family $\{f_{1},f_2,,\ldots ,f_{2k}\}$ on
$G$ is a $\gamma _{kR}(G)$-function. This implies that $f_{i}(x)=1$ for each $x\in V(K_n)$. Hence
$f_{1}\equiv f_2\equiv \cdots \equiv f_{2k}$ which is a contradiction. Thus $d_{R}^{k}(K_{n})\leq 2k-1$.

In the special case $k\ge 2$ and $2k-2\le n\le 2k-1$, Observation \ref{obs2} shows that $d_R^k(K_n)\ge 2k-1$
and so $d_R^k(K_n)=2k-1$.
\end{proof}

In view of Proposition \ref{pq} and Theorem \ref{gammast} we obtain the next upper bound for
the Roman $(k,k)$-domatic number of complete bipartite graphs.
\begin{corollary}
\label{Kpq} \emph{Let $K_{p,q}$ be the complete bipartite graph of order $p+q$ such that
$q\ge p\ge 1$, and let $k$ be a positive integer. Then $d_{R}^k(K_{p,q})\le 2k$ if $p<k$ or $q=p=k$,
$d_{R}^k(K_{p,q})\le \frac{2k(p+q)}{k+p}$ if  $p+q\ge 2k+1$ and $k\le p\le 3k$ and
$d_{R}^k(K_{p,q})\le \frac{p+q}{2}$ if $p\ge 3k$.}
\end{corollary}

Let $k\ge 1$ and $t\ge 3$ be two integers, and let $X=\{u_1,u_2,\ldots,u_{tk}\}$ and
$Y=\{v_1,v_2,\ldots,v_{tk}\}$ be the partite sets of the complete bipartite graph $K_{p,q}$ with
$p=q=kt$. For $1\leq i\leq tk$, define $f_{i}:V(K_{p,q}))\rightarrow \{0,1,2\}$ by
\begin{equation*}
f_{i}(u_{i})=f_{i}(u_{i+1})=\ldots =f_{i}(u_{i+k-1})= f_{i}(v_{i})=f_{i}(v_{i+1})=\ldots =f_{i}(v_{i+k-1})= 2
\end{equation*}
and $f_{i}(x)=0$ otherwise, where the indices are taken modulo $tk$.
It is a simple matter to verify that $\{f_{1},f_2,\ldots ,f_{tk}\}$ is an R$(k,k)D$ family on $K_{p,q}$
and hence $d_{R}^{k}(K_{p,q})\geq tk$. Using Corollary \ref{Kpq} for $p=q=tk\ge 3k$, we obtain
$d_{R}^{k}(K_{p,q})=tk=\frac{p+q}{2}$, and therefore Corollary \ref{Kpq} is sharp in that case.

\begin{theorem}
\label{c1} \emph{If $G$ is a graph of order $n\ge 2$, then
\begin{equation}  \label{eq9}
\gamma_{kR}(G)+d_{R}^k(G)\le n+2k
\end{equation}
with equality if and only if $\gamma_{kR}(G)=n$ and $d_{R}^k(G)=2k$ or $%
\gamma_{kR}(G)=2k$ and $d_{R}^k(G)=n$.}
\end{theorem}

\begin{proof}
If $d_{R}^{k}(G)\leq 2k-1$, then obviously $\gamma _{kR}(G)+d_{R}^{k}(G)\leq
n+2k-1$. Let now $d_{R}^{k}(G)\geq 2k$. If $\gamma _{kR}(G)\geq 2k$, Theorem %
\ref{gammast} implies that $d_{R}^{k}(G)\leq n$. According to Theorem \ref%
{gammast}, we obtain
\begin{equation}
\gamma _{kR}(G)+d_{R}^{k}(G)\leq \frac{2kn}{d_{R}^{k}(G)}+d_{R}^{k}(G).
\label{eq4}
\end{equation}%
Using the fact that the function $g(x)=x+(2kn)/x$ is decreasing for $2k\leq
x\leq \sqrt{2kn}$ and increasing for $\sqrt{2kn}\leq x\leq n$, this
inequality leads to the desired bound immediately.

Now let $\gamma_{kR}(G)\le 2k-1$. Since $\min\{n,\gamma_k(G)+k\}\le
\gamma_{kR}(G)$, we deduce that $\gamma_{kR}(G)=n$. According to Theorem \ref%
{gammast}, we obtain $d_{R}^k(G)\le 2k$ and hence $d_{R}^k(G)= 2k$. Thus
\begin{equation*}
\gamma_{kR}(G)+d_{R}^k(G)= n+2k.
\end{equation*}

If $\gamma_{kR}(G)=n$ and $d_{R}^k(G)=2k$ or $\gamma_{kR}(G)=2k$ and $d_{R}^k(G)=n$%
, then obviously $\gamma_{kR}(G)+d_{R}^k(G)=n+2k$.

Conversely, let equality hold in (\ref{eq9}). It follows from (\ref{eq4})
that
\begin{equation*}
n+2k=\gamma_{kR}(G)+d_{R}^k(G)\le\frac{2kn}{d_{R}^k(G)}+d_{R}^k(G)\le
n+2k,
\end{equation*}
which implies that $\gamma_{kR}(G)=\frac{2kn}{d_{R}^k(G)}$ and $d_{R}^k(G)=2k
$ or $d_{R}^k(G)=n$. This completes the proof.
\end{proof}

The special case $k=1$ of the next result can be found in \cite{SV}.

\begin{theorem}
\label{kdelta} \emph{For every graph $G$ and positive integer $k$,
\begin{equation*}
d_{R}^k(G)\le \delta(G)+2k.
\end{equation*}
Moreover, the upper bound is sharp.}
\end{theorem}

\begin{proof}
If $d_{R}^k(G)\le 2k$, the result is immediate. Let now $d_{R}^k(G)\ge 2k+1$
and let $\{f_1, f_2,\ldots,f_d\}$ be an R$(k,k)$D family on $G$ such that $d
= d_{R}^k(G)$. Assume that $v$ is a vertex of minimum degree $\delta(G)$.
Clearly the equality $\sum_{u\in N[v]}f_i(u)=1$ holds for at most $2k$
indices $i\in \{1,2,\ldots,d\}$, say $i=1,\ldots,2k$, if any. In this case $%
f_i(v)=1$ and $f_i(u)=0$ for each $u\in N(v)$ and $i=1,\ldots,2k$. It
follows that for $2k+1\le i\le d$, $f_i(v)=0$ and thus $\sum_{u\in
N[v]}f_i(u)\ge 2k$ for $2k+1\le i\le d$. Altogether we obtain
\begin{eqnarray*}
2k(d-2k)+2k & \le & \sum_{i=1}^d\sum_{u\in N[v]}f_i(u) \\
& = & \sum_{u\in N[v]}\sum_{i=1}^df_i(u) \\
& \le & \sum_{u\in N[v]} 2k \\
& = & 2k(\delta(G)+1)
\end{eqnarray*}
This inequality chain leads to the desired bound.\newline

To prove the sharpness of this inequality, let $G_i$ be a copy of $K_{k^3+(2k+1)k}$ with vertex set
$V(G_i)=\{v^i_1,v^i_2,\ldots,v^i_{k^3+(2k+1)k}\}$ for $1\le i\le k$ and let
the graph $G$ be obtained from $\cup_{i=1}^kG_i$ by adding a new vertex $v$
and joining $v$ to each $v_1^i,\ldots,v_k^i$. Define the Roman $k$%
-dominating functions $f_i^s,h_l$ for $1\le i\le k, 0\le s\le k-1$ and $1\le
l\le 2k$ as follows:
\begin{equation*}
f_i^s(v_1^i)=\cdots=f_i^s(v_k^i)=2,\;f_i^s(v_{(i-1)k^2+(s+1)k+1}^j)=%
\cdots=f_i^s(v_{(i-1)k^2+(s+1)k+k}^j)=2
\end{equation*}
\begin{equation*}
\;\mathrm{if}\; j\in \{1,2,\ldots,k\}-\{i\} \; \mathrm{and}\;f_i^s(x)=0\;\;
\mathrm{otherwise}
\end{equation*}
and for $1\le l\le 2k$,
\begin{equation*}
h_{l}(v)=1, h_{l}(v^i_{k^3+lk+1})=\ldots=h_{l}(v^i_{k^3+lk+k})=2\;(1\le i\le
k),\;\mathrm{and}\; h_{l}(x)=0\;\mathrm{otherwise}.
\end{equation*}

It is easy to see that $f_i^s$ and $g_l$ are Roman $k$-dominating function
on $G$ for each $1\le i\le k,0\le s\le k-1,1\le l\le 2k$ and $%
\{f_i^s,g_l\mid 1\le i\le k, 0\le s\le k-1 \;\mathrm{and}\;1\le l\le 2k\}$
is a Roman $(k,k)$-dominating family on $G$. Since $\delta(G)=k^2$, we have $%
d_R^k(G)=\delta(G)+2k$.
\end{proof}


For regular graphs the following improvement of Theorem \ref{kdelta} is
valid.

\begin{theorem}\label{reg}
\label{kreg} \emph{Let $k$ be a positive integer. If $G$ is a $\delta(G)$%
-regular graph, then
\begin{equation*}
d_{R}^k(G)\le\max\{2k-1,\delta(G)+k\}\le \delta(G)+2k-1.
\end{equation*}
}
\end{theorem}

\begin{proof}
If $k>\Delta(G)=\delta(G)$ then by Observation \ref{obs}, $d_{R}^k(G)\le 2k-1
$ and the desired bound is proved. If $k\le\Delta(G)$, then it follows from
Corollary \ref{Delta1} that
\begin{equation*}
d_{R}^k(G)\le \delta(G)+k,
\end{equation*}
and the proof is complete.
\end{proof}

As an application of Theorems \ref{kdelta} and \ref{reg}, we
will prove the following Nordhaus-Gaddum type result.

\begin{theorem}
\label{knord} \emph{Let $k\ge 1$ be an integer. If $G$ is a graph of order $n
$, then
\begin{equation}
d_{R}^k(G)+d_{R}^k(\overline{G})\le n+4k-2,
\end{equation}
with equality only for graphs with $\Delta(G)-\delta(G)=1$. }
\end{theorem}

\begin{proof}
It follows from Theorem \ref{kdelta} that
$$d_{R}^k(G)+d_{R}^k(\overline{G})\le (\delta(G)+2k)+(\delta(\overline{G})+2k)=
(\delta(G)+2k)+(n-\Delta(G)-1+2k).$$
If $G$ is not regular, then $\Delta(G)-\delta(G)\ge 1$, and hence this inequality implies the desired bound
$d_{R}^k(G)+d_{R}^k(\overline{G})\le n+4k-2$. If $G$ is $\delta(G)$-regular, then we deduce from Theorem \ref{reg} that
$$d_{R}^k(G)+d_{R}^k(\overline{G})\le (\delta(G)+2k-1)+(\delta(\overline{G})+2k-1)=
n+4k-3,$$
and the proof of the Nordhaus-Gaddum bound (6) is complete.
\end{proof}

\begin{corollary}
\emph{(\cite{SV}) For every graph $G$ of order $n$,
\begin{equation*}
d_{R}(G)+d_{R}(\overline{G})\le n+2,
\end{equation*}
with equality only for graphs with $\Delta(G)=\delta(G)+1$. }
\end{corollary}

For regular graphs we prove the following Nordhaus-Gaddum inequality.

\begin{theorem}
\label{regnord} \emph{Let $k\ge 1$ be an integer. If $G$ is a $\delta$-regular graph of order $n$, then
\begin{equation}
d_{R}^k(G)+d_{R}^k(\overline{G})\le\max\{4k-2,n+2k-1,n+3k-2-\delta,3k+\delta-1\}.
\end{equation}}
\end{theorem}

\begin{proof}
Let $\delta(G)=\delta$ and $\delta(\overline G)=\overline\delta$. We distinguish four cases.

If $k\ge\delta+1$ and $k\ge\overline\delta+1$, then it follows from Observation \ref{obs} that
$$d_{R}^k(G)+d_{R}^k(\overline{G})\le (2k-1)+(2k-1)=4k-2.$$

If $k\le\delta$ and $k\le\overline\delta$, then Corollary \ref{Delta1} implies that
$$d_{R}^k(G)+d_{R}^k(\overline{G})\le(\delta+k)+(\overline\delta+k)=\delta+2k+n-1-\delta=n+2k-1.$$

If $k\ge\delta+1$ and $k\le\overline\delta$, then we deduce from Observation \ref{obs} and
Corollary \ref{Delta1} that
$$d_{R}^k(G)+d_{R}^k(\overline{G})\le(2k-1)+(\overline\delta+k)=3k-1+n-1-\delta=n+3k-2-\delta.$$

If $k\le\delta$ and $k\ge\overline\delta+1$, then Observation \ref{obs} and Corollary \ref{Delta1}
lead to
$$d_{R}^k(G)+d_{R}^k(\overline{G})\le(\delta+k)+(2k-1)=3k+\delta-1.$$
This completes the proof.
\end{proof}

If $G$ is a $\delta$-regular graph of order $n\ge 2$, then Theorem \ref{regnord} leads to the
following improvement of Theorem \ref{knord} for $k\ge 2$.

\begin{corollary}
\emph{Let $k\ge 2$ be an integer. If $G$ is a $\delta$-regular graph of order $n\ge 2$, then
$$d_{R}^k(G)+d_{R}^k(\overline{G})\le n+4k-4.$$}
\end{corollary}

\end{document}